\documentclass[a4paper,12pt]{amsart}
\usepackage{a4wide}
\usepackage{amsmath,amssymb,amsthm}
\usepackage{mathrsfs}

\usepackage{enumerate}
\usepackage{amsfonts}
\newcommand{\ud}[0]{\,\mathrm{d}}

\newcommand{\abs}[1]{\left|#1\right|}
\newcommand{\Babs}[1]{\Big|#1\Big|}
\newcommand{\norm}[2]{\left|#1\right|_{#2}}

\newcommand{\Norm}[2]{\left\|#1\right\|_{#2}}
\newcommand{\BNorm}[2]{\Big\|#1\Big\|_{#2}}
\newcommand{\bnorm}[2]{\Big|#1\Big|_{#2}}

\newcommand{\pair}[2]{\left\langle #1,#2 \right\rangle}

\newcommand{\supp}[0]{\operatorname{supp}}

\newcommand{\sign}[0]{\operatorname{sgn}}
\newcommand{\BMOtwo}[0]{\operatorname{BMO}}

\newcommand\R{\mathbb{R}}
\newcommand\C{\mathbb{C}}

\newcommand\Z{\mathbb{Z}}
\newcommand\T{\mathbf{T}}

\newcommand{\la}{\lambda}
\newcommand{\eps}{\varepsilon}

\newcommand{\Exp}[0]{\mathbb{E}}

\newcommand{\radem}[0]{{\bf r}}

\newcommand{\Rad}[0]{\operatorname{Rad}}

\newcommand{\UMD}[0]{\mathit{UMD}}
\newcommand{\LPR}[1]{\mathit{LPR}_{#1}}
\newcommand{\LPReq}[1]{\mathit{LPR}_{#1}^=}
\newcommand{\BMO}[0]{\mathit{BMO}}

\theoremstyle{plain}
\newtheorem{theorem}{Theorem}[section]
\newtheorem{lemma}[theorem]{Lemma}
\newtheorem{corollary}[theorem]{Corollary}

\theoremstyle{definition}

\numberwithin{equation}{section}

\vtop to0pt{\hbox{\textsc{\footnotesize To appear in The Royal Society of Edinburgh Proc. A (Mathematics)}
 }}

\title[The Littlewood--Paley--Rubio de Francia property]
{The Littlewood--Paley--Rubio de Francia property \\
of a Banach space for the case of equal intervals}
\author[T. P. Hyt\"onen]{T. P. Hyt\"onen${}^1$}
\address{Department of Mathematics and Statistics\\ University of Helsinki\\ Gustaf H\"allstr\"omin katu 2b\\ 00014 Helsinki\\ Finland}
\email{tuomas.hytonen@helsinki.fi}
\author[J. L. Torrea]{J. L. Torrea${}^2$}
\author[D. V. Yakubovich]{D. V. Yakubovich${}^2$}
\address{Departamento de Matem\'aticas\\
Universidad Aut\'onoma de Madrid \\
Cantoblanco \\ 28049 Madrid\\ Spain}
\email{joseluis.torrea@uam.es}
\email{dmitry.yakubovich@uam.es}

\date{\today}

\begin{document}
\footnotetext[1]{The first author is supported by the Academy of Finland,
project 114374 ``Vector-valued singular integrals''.}
\footnotetext[2]{The second and the third authors acknowledge
the support of the Grant MTM2005-08359-C03-01 by the
European Fund FEDER and
Ministry of Science and Technology of Spain
(MEC). The third author also acknowledges the support of
MEC under the Ram\'on y
Cajal Programme (2002).}

\maketitle


\section{Introduction and results}

Let $X$ be a Banach space, with the norm denoted as $\abs{\cdot}_{X}$.
Put $\T=[0,1)$, which we identify with the
additive group $\R/\Z$. Then the
characters of $\T$ are the functions
$e_k(t)=e^{i2\pi k t}$, $k\in\Z$. Let
$L^p_{X}(\T)$ be the Bochner space
of $X$-valued $p$-integrable functions on $\T$. For intervals
$I\subset\Z$, the corresponding partial sums of the Fourier series
of $f$ are denoted by
\begin{equation}\label{eq:SI}
  S_I f:=\sum_{k\in I}\hat{f}(k)e_k.
\end{equation}

In the case when $X=\C$, useful inequalities of the type
\begin{equation}\label{eq:LP}
  \BNorm{\Big(\sum_j\abs{S_{I_j}f}^2\Big)^{1/2}}{L^p(\T)}\leq C\Norm{f}{L^p(\T)}
\end{equation}
are well known. When $\{I_j\}_{j\in\Z}$ is
the collection of
dyadic intervals, i.e., $I_0=\{0\}$ and
$I_j=\sign(j)[2^{|j|-1},2^{|j|})$ for $\abs{j}>0$, the estimate
\eqref{eq:LP} is the classical Littlewood--Paley inequality which
is valid (as well as the reverse estimate with $\geq$ in place of
$\leq$) for all $p\in(1,\infty)$. If the $I_j$ are disjoint
intervals of equal length, then \eqref{eq:LP} holds if and only if
$p\in[2,\infty)$; this was first proved by L.~Carleson
\cite{Carleson}, and then in different ways by both A.~C\'ordoba
\cite{Cordoba} and J.~L.\ Rubio de Francia \cite{RdF:83}, who also
finally showed that, for the same range of exponents,
the analogue of
\eqref{eq:LP} for the case of
the real line is actually true for an arbitrary collection of
disjoint subintervals $I_j$ of $\R$, see \cite{RdF:85}. As we will
explain later,
the versions of
\eqref{eq:LP} for the unit circle and for the real line are
equivalent.

By Hin\v{c}in's inequality, there is a two-sided comparison
\newpage
\[
  \Big(\sum_j\abs{a_j}^2\Big)^{1/2}\eqsim_p
  \Big(\Exp\Babs{\sum_j\radem_j a_j}^p\Big)^{1/p}
\]
for all $0<p<\infty$, where the $\radem_j$ are independent
Rademacher random variables on some probability space $\Omega$
(i.e., they take the values $+1$ and $-1$ with equal probability
$1/2$), and $\Exp$ is the mathematical expectation. Thus, an
equivalent formulation of \eqref{eq:LP} reads
\begin{equation*}
  \BNorm{\Exp\Babs{\sum_j\radem_j S_{I_j}f}}{L^p(\T)}\leq C\Norm{f}{L^p(\T)}.
\end{equation*}
\noindent This has proven to be a useful formulation when looking for
analogues
of the Littlewood--Paley type estimates \eqref{eq:LP} in
the Bochner spaces $L^p_X(\T)$. In particular, it is known from
the work of J.~Bourgain \cite{Bourgain:86} that
\begin{equation}\label{eq:LPB}
  \BNorm{\Exp\bnorm{\sum_j\radem_j S_{I_j}f}{X}}{L^p(\T)}
  \leq C\Norm{f}{L^p_X(\T)}
\end{equation}
holds for the dyadic system of intervals if and only if
$p\in(1,\infty)$ and $X$ is a Banach space with the
unconditionality property of martingale differences (UMD). In
fact, already the uniform boundedness of the partial sum
projections \eqref{eq:SI} on $L^p_X(\T)$ is equivalent to
$X\in\UMD$, since both are equivalent to the boundedness of the
Hilbert transform $H=-i[S_{[0,\infty)}-S_{(-\infty,0)}]$ on 
$L^p_X(\T)$; cf.~\cite{Bourgain:83,Bourgain:86,Burkholder}.

As for vector-valued analogues of \eqref{eq:LP} for other systems
of intervals, E.~Berkson, T.~A.\ Gillespie and J.~L.\ Torrea
\cite{BGT} have introduced the following terminology. The space
$X$ is said to have the $\LPR{p}$ (Littlewood--Paley--Rubio)
property if \eqref{eq:LPB} holds for an arbitrary finite disjoint
collection of intervals $I_j$ with a constant independent of the
intervals and $f\in L^p_X(\T)$. To be precise, the corresponding
property in \cite{BGT} was considered in the context of
$L^p_X(\R)$ that is
\begin{equation}\label{eq:LPBr}
  \BNorm{\Exp\bnorm{\sum_j\radem_j S_{J_j}^{\R}f}{X}}{L^p(\R)}
  \leq C\Norm{f}{L^p_X(\R)}
\end{equation}
for disjoint intervals $J_j\subset\R$. Here
\begin{equation}\label{eq:Dm2}
S^\R_{J_j} f(x) = \int_{I_j} \hat{f}(\lambda) e^{i2\pi\lambda x} \ud\lambda,
\end{equation}
where
\begin{equation}\label{eq:cont_Four}
\hat{f}(\lambda) = \int_\R f(x) e^{-i2\pi\lambda x} \ud x
\end{equation}
is the Fourier transform of $f$.
It was pointed out in \cite{GT} that this is
equivalent to $\LPR{p}$ as  defined  by \eqref{eq:LPB}.

From the scalar-valued
results and a restriction to a one-dimensional subspace it is
clear that property $\LPR{p}$ is meaningful only for
$p\in[2,\infty)$. Moreover, by using the characterization of
UMD by the uniform $L^p_X(\T)$-boundedness of the $S_I$, it
follows that $\LPR{p}\Rightarrow\UMD$.

There is also a connection to the type of the Banach space, which
is defined as follows. The space $X$ has type $t$ if and only if
\begin{equation*}
  \Exp\bnorm{\sum_j\radem_j x_j}{X}\leq C\Big(\sum_j\norm{x_j}{X}^t\Big)^{1/t}.
\end{equation*}
This condition is trivial (by the triangle inequality) for $t=1$;
it becomes more restrictive with increasing $t$ and an
impossibility for $t>2$. One defines
\begin{equation*}
  p(X):=\sup\{t\in[1,2]:X\text{ has type }t\}.
\end{equation*}
It was shown in \cite{BGT} that
\begin{equation*}
  X\in LPR_p\quad\Rightarrow\quad p(X)=2.
\end{equation*}

Besides these observations, the $\LPR{p}$ property remains quite
mysterious. For instance, it is not known whether there is an
implication between $\LPR{p}$ and $\LPR{q}$ for two different
$p,q\in[2,\infty)$, and basically the only known examples of
spaces with this property are classical $L^p$ spaces. On the other
hand, the $\LPR{p}$ property has some useful implications
concerning multipliers of $X$-valued Fourier series,
cf.~\cite{Lacey, HP}. We refer to \cite{CowlingTao, KislyakovParilov}
for related results.

In this note, our aim is to gain understanding of this property by
studying the special case of equal-length intervals. Thus we say
that the Banach space $X$ has the $\LPReq{p}(\T)$ (respectively
$\LPReq{p}(\R)$)  property if \eqref{eq:LPB} (respectively
\eqref{eq:LPBr})  holds with a fixed constant $C$ for all
$f\in L^p_X(\T)$ (respectively $f\in L^p_X(\R)$) and all finite
sequences of non-overlapping intervals
$I_j$ of equal length. We also introduce a
number of other variants of this condition after recalling some
notation.

Let $w\in L^1(\T)$ be a weight function. We denote by
$L^p_X(\T,w)$ the weighted $L^p$ space, i.e., the $L^p$ space with
respect to the measure $w(t)\ud t$. In the sequel, we are
particularly interested in weights in the Muckenhoupt $A_q$
classes.

Let $M:[0,\infty)\to[0,\infty)$ be a convex function
such that $M(0)=0$ and $\lim_{t\to\infty} M(t)=\infty$.
Recall that the Orlicz space
$L_M(\T)$ consists of those functions
$g\in L^1(\T)$, which satisfy
$\int_{\T}M(\lambda\abs{g(t)})\ud t<\infty$ for some
$\lambda>0$. It is a Banach space with the norm
\[
   \Norm{g}{L_M(\T)}
   =\inf\big\{ \lambda>0: \quad
    \int_{\T}M(\abs{g(t)}/\lambda)\ud t<1 \big\}.
\]
Its $X$-valued version is denoted with the subscript $X$, as usual.
In the case of the function $M_0(t)=\exp(t)-1$, we denote
$L_{M_0}(\T)=\exp L(\T)$.

In the following conditions, it is always understood that there
should be a fixed $C$ such that the given inequality holds
whenever $I_j$ are allowed to be any finite collection of
non-overlapping intervals of equal length, and $f$ an arbitrary
function in the space indicated by the right-hand side of the
inequality. We say that $X$ satisfies $\LPReq{p}(\T,w)$ if
\begin{equation}\label{eq:LPweight}
  \BNorm{\Exp\bnorm{\sum_j\radem_j S_{I_j} f}{X}}{L^p(\T,w)}
  \leq C\Norm{f}{L^p_X(\T,w)},
\end{equation}
we say that $X$ satisfies $\LPReq{\infty}(\T)$ if
\begin{equation}\label{eq:LPexp}
  \BNorm{\Exp\bnorm{\sum_j\radem_j S_{I_j}f}{X}}{\exp L(\T)}
  \leq C\Norm{f}{L^{\infty}_X(\T)};
\end{equation}
and finally $X$ satisfies $\LPReq{\infty,1}(\T)$ if
\begin{equation}\label{eq:LPweak}
  \BNorm{\Exp\bnorm{\sum_j\radem_j S_{I_j}f}{X}}{L^1(\T)}
  \leq C\Norm{f}{L^{\infty}_X(\T)}.
\end{equation}
We shall also say that  $X$ satisfies $\LPReq{p}(\R,w)$ if
an analogue of formula \eqref{eq:LPweight} holds, with
$\T$ replaced by $\R$.

Out main result is that all these equal-interval conditions are
equivalent, and that moreover they admit a simple characterization in
terms of well-established Banach space properties. This also
sharpens the necessary conditions for the original $\LPR{p}$
property from the earlier results quoted above.

\begin{theorem} \label{toro}{\bf (Case of $\T$.)}
Let $X$ be a Banach space. The following conditions are equivalent:
\begin{itemize}
  \item $X$ is a UMD space with type $2$,
  \item $X$ has $\LPReq{p}(\T,w)$ for all $p\in[2,\infty)$ and all $w\in A_{p/2}$,
  \item $X$ has $\LPReq{p}(\T)$ for some $p\in[2,\infty)$,
  \item $X$ has $\LPReq{\infty}(\T)$,
  \item $X$ has $\LPReq{\infty,1}(\T)$.
\end{itemize}
\end{theorem}

We also have the following result.

\begin{theorem}\label{real} {\bf (Case of $\R$.)}
Let $X$ be a Banach space. The following conditions are equivalent:
\begin{itemize}
  \item $X$ is a UMD space with type $2$,
  \item $X$ has $\LPReq{p}(\R,w)$ for all $p\in[2,\infty)$ and all $w\in A_{p/2}$,
  \item $X$ has $\LPReq{p}(\R)$ for some $p\in[2,\infty)$.
\end{itemize}
\end{theorem}

The proof of  Theorem \ref{toro} will comprise the main part of our arguments.
Since we have the following containments, with bounded inclusion maps,
\[
  L^{\infty}(\T)\subset\exp L(\T)\subset L^p(\T)\subset L^1(\T),\qquad p\in(1,\infty),
\]
it is immediate that $\LPReq{p}(\T)\Rightarrow\LPReq{\infty,1}(\T)$ for
all $p\in[2,\infty]$. Hence it suffices to show that
$\LPReq{\infty,1}(\T)$ implies UMD and type $2$, and UMD
together with type $2$ imply $\LPReq{p}(\T,w)$ for all
$p\in[2,\infty)$ and $w\in A_{p/2}$, as well as $\LPReq{\infty}(\T)$.

Theorem \ref{real} will be derived from Theorem \ref{toro}.
In \cite{GT}, a proof is provided in order
to see that
property $\LPR{p}$  for $\R$ implies
the property $\LPR{p}$  for $\T$.
This proof can be adapted directly, and
one gets that $\LPReq{p}(\R)$ implies $\LPReq{p}(\T)$.
The converse will follow from Lemma \ref{paso}.

Finally, in Section \ref{miscelanea} we will
comment on the validity of $L^\infty$--$\BMOtwo$ results.

\section{The necessity of UMD and type $2$} 

Let us first deal with the UMD property. This already follows from
\eqref{eq:LPweak} with an arbitrary single interval $I$; the
estimate says that
\[
  \Norm{S_I f}{L^1_X(\T)}\leq C\Norm{f}{L^{\infty}_X(\T)},
\]
which implies a similar inequality with the Hilbert transform
$H=-i[S_{[0,\infty)}-S_{(-\infty,0)}]$ in place of $S_I$. Using
the description of the Hardy space $H^1_X(\T)$ in terms of
$L^{\infty}_X(\T)$-atoms, it follows that $H$ is bounded from
$H^1_X(\T)$ to $L^1_X(\T)$. This implies its boundedness on
$L^p_X(\T)$ by O.~Blasco \cite{Blasco}, which in turn implies UMD
by Bourgain~\cite{Bourgain:83}.

In what follows, we will use the following well-known result.
\begin{lemma}[Kahane's contraction principle, see
\cite{KunstmannWeis}, Proposition 2.5]\label{lem:Kahane-contr}
For any se\-quen\-ces $\{a_n\}\subset\C$,
$\{x_n\}\subset X$  and any $p\in [1,\infty) $,
\[
\Exp \abs{\sum_n a_n\radem_n x_n}_X^p
\leq
\big(2\sup_n |a_n|\big)^p\cdot
\Exp \abs{\sum_n \radem_n x_n}_X^p
\]
\end{lemma}
%
%

We turn to the type~$2$ property, starting from the following
bootstrapping of the original estimate.

\begin{lemma}\label{lem:LPbootstrap}
Let $X$ have $\LPReq{\infty,1}(\T)$. Let $\{I_{jk}\}_{j,k}$ be a
two-parameter family of equal-length intervals such that for every
fixed $j$, the family $\{I_{jk}\}_{k}$ is finite and consists of
disjoint intervals. Then
\begin{equation}\label{eq:LPbootstrap}
  \BNorm{\Exp\bnorm{\sum_{j,k}\radem_{jk}S_{I_{jk}}f_j}{X}}{L^1(\T)}
  \leq C\BNorm{\max_{\eps_j=\pm 1}\bnorm{\sum_j\eps_j f_j}{X}}{L^{\infty}(\T)}
\end{equation}
for all finite families of $f_j\in L^{\infty}_X(\T)$.
\end{lemma}

\begin{proof}
We may assume by approximation that the Fourier transforms
$\hat{f}_j$ are compactly supported. Then we may choose intervals
$J_j$ such that $J_j\supseteq\supp\hat{f}_j$, and also
$J_j\supseteq I_{jk}$ for all $k$. Next we choose integers $N_j$
such that the translated intervals $J_j+N_j$ are all disjoint.
Then, manipulating the random sums with the help of Kahane's
contraction principle
\begin{equation*}\begin{split}
  &\BNorm{\Exp\bnorm{\sum_{j,k}\radem_{jk}S_{I_{jk}}f_j}{X}}{L^1(\T)}
  =\BNorm{\Exp\bnorm{\sum_{j,k}\radem_{jk}e_{-N_j}S_{I_{jk}+N_j}e_{N_j}f_j}{X}}{L^1(\T)} \\
  &\leq 2\BNorm{\Exp\bnorm{\sum_{j,k}\radem_{jk}S_{I_{jk}+N_j}e_{N_j}f_j}{X}}{L^1(\T)}
  =2\BNorm{\Exp\bnorm{\sum_{j,k}\radem_{jk}S_{I_{jk}+N_j}\sum_i e_{N_i}f_i}{X}}{L^1(\T)} \\
  &\leq 2C\BNorm{\sum_i e_{N_i}f_i}{L^{\infty}_X(\T)}
   \leq 4C\BNorm{\max_{\eps_i=\pm 1}\bnorm{\sum_j\eps_j f_j}{X}}{L^{\infty}(\T)},
\end{split}\end{equation*}
where, in the second equality, we used the fact that
$S_{I_{jk}+N_j}e_{N_i}f_i=0$ if $j\neq i$.
\end{proof}

\begin{corollary}
Let $X$ have $\LPReq{\infty,1}(\T)$. Let $x_j\in X$ and $\phi_j\in
L^{\infty}(\T)$
for $j=1,\ldots,N$. Then
\begin{equation*}
  \BNorm{\Exp\bnorm{\sum_{j=1}^N\radem_j\Norm{\phi_j}{L^2(\T)}x_j}{X}}{L^1(\T)}
  \leq C\BNorm{\max_{\eps_j=\pm 1}\bnorm{\sum_{j=1}^N\eps_j\phi_j x_j}{X}}{L^{\infty}(\T)}.
\end{equation*}
\end{corollary}

\begin{proof}
We apply Lemma~\ref{lem:LPbootstrap} with $f_j=\phi_j\,x_j$ and
$I_{jk}=\{k\}$, and consider the limit where the collection
$I_{jk}$ covers all of $\Z$ for each $j$. Then
\eqref{eq:LPbootstrap} reads
\[
  \BNorm{\Exp\bnorm{\sum_{j,k}\radem_{jk}\hat{\phi_j}(k)x_j}{X}}{L^1(\T)}
  \leq C\BNorm{\max_{\eps_j=\pm 1}
                  \bnorm{\sum_j\eps_j \phi_j\,x_j}{X}}{L^{\infty}(\T)}.
\]

We already know that $X$ is a UMD space, hence has finite cotype.
In this situation, the expectations of norms of random sums of
vectors of $X$ with Rademacher coefficients $\radem_j$ are
comparable to similar expression with independent standard
complex Gaussian random variables $\gamma_j$ in place of $\radem_j$; see
e.g.~\cite{DJT}, 12.27. Thus
\[\begin{split}
  \BNorm{\Exp\bnorm{\sum_j\sum_k\radem_{jk}\hat{\phi_j}(k)x_j}{X}}{L^1(\T)}
  &\eqsim\BNorm{\Exp\bnorm{\sum_j\Big(\sum_k\gamma_{jk}\hat{\phi_j}(k)\Big)x_j}{X}}{L^1(\T)} \\
\end{split}\]
Observe that
$\Gamma_j:=\sum_k\gamma_{jk}\hat{\phi_j}(k)$ are centered
independent Gaussian random variable with variances
$\Exp|\Gamma_j|^2=\sum_k|\hat{\phi_j}(k)|^2=\Norm{\phi_j}{L^2(\T)}^2$.
Hence
\[\begin{split}
  \BNorm{\Exp\bnorm{\sum_j\sum_k\radem_{jk}\hat{\phi_j}(k)x_j}{X}}{L^1(\T)}
  &\eqsim\BNorm{\Exp\bnorm{\sum_j \gamma_j\Norm{\phi_j}{L^2(\T)}x_j}{X}}{L^1(\T)}.
\end{split}\]
After changing the $\gamma_j$ back to $\radem_j$, we have proved the assertion.
\end{proof}

We are finally in a position to establish
that the property  $\LPReq{\infty,1}(\T)$
implies type~$2$. We apply the Corollary
with $\phi_j=1_{[(j-1)/N,j/N)}$ and $x_j$ unit vectors in $X$. The result is
\[
  \Exp\bnorm{\sum_{j=1}^N\radem_j x_j}{X}\leq CN^{1/2},\qquad
  \norm{x_1}{X}=\ldots=\norm{x_N}{X}=1.
\]
This is the type~$2$ estimate for vectors of equal length. However, by an observation made by R.~C.\ James \cite{James}, this already implies the full type~$2$ condition, and we are done.

\section{Proof of the Littlewood--Paley estimate for equal intervals. Case
of the Unit Circle}

Our vector-valued proof follows the approach of Rubio de Francia
\cite{RdF:83} from the scalar-valued case, although he dealt with
the analogous estimate on $\R$ instead of $\T$. The first step
consists in replacing the spectral projections $S_I$ by some nicer
approximations. Thus, let $m_j$ be functions such that
$m_j|I_j\equiv 1$, and let
\[
  T_{m_j}f:=\sum_{k\in\Z}m_j(k)\hat{f}(k)e_k
\]
be the corresponding Fourier multiplier operators. Then
\[
  \sum_j\radem_j S_{I_j}f
  =\sum_j\radem_j S_{I_j}T_{m_j}f.
\]

\begin{lemma}\label{lem:redToSmooth}
Let $X$ be a UMD space.
Let $p\in(1,\infty)$ and $w\in A_p$.
Then, for any family of non-intersecting subintervals $I_j$ of $\Z$
(of arbitrary lengths) and any family $\{ g_j\}$
of functions in  $L^p_X(\T)$,
there holds
\begin{equation}\label{eq:weightedRed}
  \BNorm{\Exp\bnorm{\sum_j\radem_j S_{I_j}g_j}{X}}{L^p(\T,w)}
  \leq C\BNorm{\Exp\bnorm{\sum_j\radem_j g_j}{X}}{L^p(\T,w)}
\end{equation}
and also
\begin{equation}\label{eq:inftyRed}
  \BNorm{\Exp\bnorm{\sum_j\radem_j S_{I_j}g_j}{X}}{\exp L(\T)}
  \leq C\BNorm{\Exp\bnorm{\sum_j\radem_j g_j}{X}}{L^{\infty}(\T)}.
\end{equation}
\end{lemma}

\begin{proof}
We write out the proof of \eqref{eq:inftyRed} and indicate the
changes for \eqref{eq:weightedRed}. We first note that each
$S_{I_j}$ is a linear combination of two translated Hilbert
transforms,
$S_{I_j}=\frac{i}{2}[e_{-a_j}He_{a_j}-e_{-b_j}He_{b_j}]$. By the
triangle inequality, it suffices to treat the first term. By the
contraction principle (applied for $p=1$),
\[\begin{split}
  \BNorm{\Exp\bnorm{\sum_j\radem_j e_{-a_j}H(e_{a_j}g_j)} {X}}{\exp L(\T)}
  &\leq 2\BNorm{\Exp\bnorm{\sum_j\radem_j H(e_{a_j}g_j)} {X}}{\exp L(\T)}\\
  &= 2\BNorm{H\Big(\sum_j\radem_j e_{a_j}g_j\Big)}{\exp L_{\Rad X}(\T)},
\end{split}\]
where $\Rad X$ is the Banach space formed as the closure of the
linear span of the functions $\radem_j x$, $x\in X$, in
$L^1_X(\Omega)$. This is again a UMD space (since by Kahane's
inequality, the same space with an equivalent norm would be
obtained by taking the closure in $L^2_X(\Omega)$). The same lines
are also true with the weighted $L^p$ space in place of $\exp L$.

We now make use of the John--Nirenberg inequality (valid in
arbitrary Banach spaces) and the $L^{\infty}$--$\BMO$-boundedness
of the Hilbert transform (in UMD spaces) to continue the
estimation
\[\begin{split}
  \lesssim \BNorm{H\Big(\sum_j\radem_j e_{a_j}g_j\Big)}{\BMO_{\Rad X}(\T)}
  &\lesssim \BNorm{\sum_j\radem_j e_{a_j}g_j}{L^{\infty}_{\Rad X}(\T)} \\
  &=\BNorm{\Exp\bnorm{\sum_j\radem_j e_{a_j}g_j}{X}}{L^{\infty}(\T)}
  \leq 2\BNorm{\Exp\bnorm{\sum_j\radem_j g_j}{X}}{L^{\infty}(\T)}.
\end{split}\]
In the weighted case, we may directly use the boundedness of $H$ in $L^p_X(\T,w)$, and the rest is similar with $L^p(\T,w)$ in place of $L^{\infty}(\T)$.
\end{proof}

With Lemma~\ref{lem:redToSmooth}, our task is reduced to proving that
\begin{equation}\label{eq:smoothToProve}
  \BNorm{\Exp\bnorm{\sum_j\radem_j T_{m_j}f}{X}}{L^p(\T,w)}
  \leq C\Norm{f}{L^p_X(\T,w)}
\end{equation}
for appropriate smooth majorants $m_j$ of $1_{I_j}$. To describe
our choice of the $m_j$, we recall the definition of the de la
Vall\'ee Poussin kernel (see e.g.\ \cite{Katznelson})
\[
  V_n=2K_{2n+1}-K_n,
\]
where $K_n$ is the Fej\'er kernel
\[\begin{split}
  K_n(t) &=\sum_{k=-n}^n\Big(1-\frac{\abs{k}}{n+1}\Big)e_k(t)
    =\frac{1}{n+1}\Big(\frac{\sin(n+1)\pi t}{\sin\pi t}\Big)^2.
\end{split}\]
By applying the estimates $|\sin (n+1)x|\le (n+1)|\sin x|$ for $x\in \R$
(use the Euler formulas) and $\sin x\geq 2x/\pi$ for $0\le x\le \pi/2$, one gets
\[\begin{split}
  K_n(t)  &\leq\min\Big(n+1,\frac{1}{4(n+1)t^2}\Big),\qquad\abs{t}\leq 1/2.
\end{split}\]
Thus $\hat{V}_{n-1}(k)=1$ for $\abs{k}\leq n$, and
\begin{equation}\label{eq:dlVPest}
  \abs{V_{n-1}(t)}\leq C\min(n,\frac{1}{n t^2}),\qquad\abs{t}\leq 1/2.
\end{equation}

Let $L$ stand for the common length of our intervals $I_j$. For a fixed $j$,
there
are $L+1$ consecutive integers $k$ such that
$\hat{V}_{L-1}(\cdot+k)=1$ on $I_j$. Let us pick such a $k$ which
is also divisible by $L$, and write $k=L k_j$. For uniqueness, let
us take the smaller  $k$ (and then the smaller $k_j$) if two
possibilities exist. Then the mapping $I_j\mapsto k_j$ is
one-to-one, and by reindexing the intervals if necessary, we may
assume that $k_j=j$, where
$j$ ranges over an  index set $J\subset \Z$.
 Thus we take
\[
  m_j=\hat{V}_{L-1}(\cdot+Lj), \qquad j\in J,
\]
so that
\[
  \sum_{j\in J}\radem_j T_{m_j}f(t)
  =\sum_{j\in J}\radem_j\int_{\T}e_{-Lj}(y)V_{L-1}(y)f(t-y)\ud y.
\]
For a fixed $t\in\T$, we want to estimate the norm of this object
in $\Rad X$, which we now equip with the norm of $L^2_X(\Omega)$
for convenience. Since $X$, as a UMD space, is $B$-convex, we have
$(\Rad X)^*\eqsim \Rad X^*$,
see e.g.~\cite{HW}, Section 3.
In particular, one has
\[
  \BNorm{\sum_{j\in J}\radem_j T_{m_j}f(t)}{\Rad X}
  \eqsim\sup\Big\{\Babs{\sum_{j\in J}\pair{\lambda_j}{T_{m_j}f(t)}}:
    \lambda_j\in X^*, \;\BNorm{\sum_{j\in J}\radem_j\lambda_j}{\Rad X^*}\leq 1\Big\}.
\]
In what follows, we put $\lambda_j=0$ for $j\notin J$, which allows
us to extend the summation to all indices $j\in \Z$.

We now manipulate the duality pairing appearing above:
\begin{equation}\label{eq:Glambda}\begin{split}
  G_{\lambda}f(t):=\sum_j\pair{\lambda_j}{T_{m_j}f(t)}
  &=\int_{\T}V_{L-1}(y)\pair{\sum_j e_{-Lj}(y)\lambda_j}{f(t-y)}\ud y\\
  &=:\int_{\T}V_{L-1}(y)\pair{g_{\lambda}(Ly)}{f(t-y)}\ud y,
\end{split}\end{equation}
where $g_{\lambda}$ is an $X^*$-valued function on $\R$ of period $1$,
\[
g_{\lambda}(y)=\sum_{j\in\Z} e_{-j}(y)\lambda_j.
\]

\begin{lemma}\label{lem:trigVsRad}
Let $X$ have type $2$. Then
\[
  \Norm{g_{\lambda}}{L^2_{X^*}(\T)}
  \leq C\Big(\Exp\bnorm{\sum_{j\in \Z}\radem_j\lambda_j}{X^*}^2\Big)^{1/2}.
\]
\end{lemma}

\begin{proof}
Of course, we make use of duality:
\[
  \Norm{g_{\lambda}}{L^2_{X^*}(\T)}
  =\sup\Big\{\Babs{\int_{\T}\pair{g_{\lambda}(y)}{\phi(y)}\ud y}:
   \Norm{\phi}{L^2_X(\T)}\leq 1\Big\},
\]
where
\[\begin{split}
  \Babs{\int_{\T} &\pair{g_{\lambda}(y)}{\phi(y)}\ud y}
  =\Babs{\int_{\T}\pair{\sum_j e_{-j}(y)\lambda_j}{\sum_k e_k(y)\hat{\phi}(k)}\ud y}
  =\Babs{\sum_j\pair{\lambda_j}{\hat{\phi}(j)}} \\
  &=\Babs{\Exp\pair{\sum_j\radem_j\lambda_j}{\sum_k\radem_k\hat{\phi}(k)}}
  \leq\Big(\Exp\bnorm{\sum_j\radem_j\lambda_j}{X^*}^2\Big)^{1/2}
     \Big(\Exp\bnorm{\sum_k\radem_k\hat{\phi}(k)}{X}^2\Big)^{1/2}.
\end{split}\]
Now finally we employ the assumed type~$2$ property of $X$. By a
result of P.~G.\ Dodds and F.~A.\ Sukochev (\cite{DS}, Theorem 1.3), this
guarantees that
\[
  \Big(\Exp\bnorm{\sum_k\radem_k\hat{\phi}(k)}{X}^2\Big)^{1/2}
  \leq C\BNorm{\sum_k e_k \hat{\phi}(k)}{L^2_X(\T)}=C\Norm{\phi}{L^2_X(\T)},
\]
which completes the proof.
\end{proof}

Now we can continue with the expression \eqref{eq:Glambda}. We start by estimating the de la Vall\'ee Poussin kernel according to \eqref{eq:dlVPest}:
\[\begin{split}
  \abs{G_{\lambda}f(t)}
  &\leq\int_{\T}\abs{V_{L-1}(y)}\abs{\pair{g_{\lambda}(Ly)}{f(t-y)}}\ud y \\
  &\leq\Big(\int_{\abs{y}\leq 1/L} CL+\sum_{k=1}^{\infty}
    \int_{2^{k-1}/L<\abs{y}\leq 2^k/L} CL 2^{-2k}\Big)\norm{g_{\lambda}(Ly)}{X^*}\norm{f(t-y)}{X}\ud y \\
  &\leq C\sum_{k=0}^{\infty}2^{-k}
   \Big(\frac{L}{2^k}\int_{\abs{y}\leq 2^k/L}
    \norm{g_{\lambda}(Ly)}{X^*}^2\ud y\Big)^{1/2}
   \Big(\frac{L}{2^k}\int_{\abs{y}\leq 2^k/L}
    \norm{f(t-y)}{X}^2\ud y\Big)^{1/2} \\
  &\leq C\sum_{k=0}^{\infty}2^{-k}\Big(\frac{L}{2^k}1_{[-2^k/L,2^k/L]}*\norm{f}{X}^2(t)\Big)^{1/2}
   \leq C \big(M\norm{f}{X}^2(t)\big)^{1/2}
\end{split}\]
where Lemma~\ref{lem:trigVsRad}, and the condition that $\Norm{\sum\radem_j\lambda_j}{\Rad X^*}\leq 1$, were employed in the second-to-last step, while the last estimate is clear from the definition of the Hardy--Littlewood maximal function~$M$.

For $p>2$, the estimate~\eqref{eq:smoothToProve} now follows from
the well-known boundedness of $M$ in $L^q(\T,w)$ for $w\in A_q$
and $q=p/2>1$; the case $p=\infty$ is of course trivial. For
$p=2$, we need to stop the preceding computation before estimating
by the maximal function, and then observe that the convolution
operators with functions $\abs{J}^{-1}1_J$ are uniformly bounded on $L^1(\T,w)$
for $w\in A_{2/2}=A_1$.

It might be interesting to note that out of the two key steps of
the proof, Lemma~\ref{lem:redToSmooth} only employed the UMD
assumption, whereas Lemma~\ref{lem:trigVsRad} only used type~$2$.
Indeed the proof shows that a variant of the $\LPReq{p}(\T)$ estimate,
with the sharp spectral projections replaced by those of de la
Vall\'ee Poussin type, is true in all Banach spaces with type~$2$.

\section{Proof of Theorem \ref{real}. }\label{Proofreal}

\begin{lemma}\label{paso} For any Banach space $X$, the following statements are equivalent
\begin{itemize}
\item[(i)] Inequality (\ref{eq:LPB}) holds for families of disjoint intervals  of the same length in $\Z$.
\item[(ii)] Inequality (\ref{eq:LPBr}) holds for families of disjoint intervals  of the same length in $\mathbb{R}$.

\end{itemize}

\end{lemma}
\begin{proof} The proof of $(ii) \Longrightarrow (i)$ can be got
by following  the proof of Theorem 2.5 in  \cite{GT}.

Now let us assume $(i)$, and let us prove $(ii)$.
We will use the fact that the discrete partial Fourier sums
\eqref{eq:SI} are Riemann sums of their continuous analogues
\eqref{eq:Dm2}.
Take a finite number of non-intersecting intervals
$\{J_j\}_{j=1}^m$ in $\R$
of equal lengths. Let us show first that for any $A>0$ and any $f\in L^p_X([-A,A])$,
\begin{equation}\label{eq:3A}
\|\Exp \sum_{j=1}^m \radem_j\, S_{J_j}^{\R}f\|_{L^p_X([-A,A])}
\le
C\Norm{f}{L^p_X([-A,A])}.
\end{equation}
We shall give an approximation argument.
By the assumption $(i)$,
\begin{equation} \label{eq:assump_i}
\BNorm{\Exp\bnorm{\sum_j\radem_j S_{I_j}f}{X}}{L^p([-1/2,1/2])}
\leq  C\Norm{f}{L^p_X([-1/2,1/2])}
\end{equation}
for any $f\in L^p_X([-1/2,1/2])$
and for any finite sequence $\{I_j\}$ of subintervals of
$\Z$ of the same length.
First we apply rescaling to \eqref{eq:assump_i}.
Given a function $f$ and real $\ell >0$,
we define a rescaled function $f_\ell(x) = \frac1\ell f(\frac{x}{\ell})$.
One has
\begin{eqnarray*}
\widehat{(f_{1/\ell})}(\lambda)  =
\hat{f}\,(\frac\lambda\ell).
\end{eqnarray*}
Given a finite subinterval $I$
of the discrete set $\ell^{-1} \mathbb{Z}$, define
\[
S_{I}^{(\ell)} f(x) = \frac{1}{\ell} \,\sum_{\lambda \in I}
\hat{f}(\lambda) e^{i 2\pi\lambda x}.
\]
Then $\ell I$ is a subinterval in $\mathbb{Z}$, and it is easy to see that
\[
S_{I}^{(\ell)} f(x) = \Big( S_{\ell I} (f_{1/\ell}) \Big)_\ell(x).
\]
We get from \eqref{eq:assump_i} the inequality
\begin{equation}\label{eq:dconstants}
\BNorm{\Exp\bnorm{\sum_{j=1}^m\radem_j S_{I_j}^{(\ell)}f}{X}}{L^p_X([-\ell/2, \ell/2])}
\leq  C\Norm{f}{L^p_X([-\ell/2, \ell/2])}
\end{equation}
for functions $f\in L^p_X([-\ell/2, \ell/2])$
and subintervals $I_j$ of $\ell^{-1}\mathbb{Z}$, with
the same constant $C$ as in \eqref{eq:assump_i}.

To prove \eqref{eq:3A}, take any $\ell>\max(2A,1/2|J_j|)$,
so that $[-\ell/2, \ell/2]\supset [-A, A]$.
We extend our function $f\in L^p_X([-A, A])$ to
a function in $L^p_X(\R)$ by putting
$f\equiv 0$ on $\R\setminus[-A, A]$.
It is easy to find
subintervals $I_{j,\ell}\subset J_j\cap \ell^{-1}\Z$ with the
following two properties:
\begin{enumerate}
\item For a fixed $\ell$, the intervals $I_{j,\ell}$ are disjoint and
have the same length;
\item $|J_j|-\frac 2 \ell \le |I_{j,\ell}| <|J_j| -\frac 1 \ell$.
\end{enumerate}

One easily gets from (2) an estimate
\[
\bigg|
\int_{J_j} g(\la)\,d\la
 - \frac 1 \ell \sum_{\la\in I_{j,\ell}} g(\la)
\bigg|_X
\le
\frac{\abs{J_j}}{2\ell} \;
\max_{\la\in {J_j}}\;
\big| \frac {dg(\la)}{d\la}\big|_X+
\frac 2 \ell
\max_{\la\in {J_j}}\; |g(\la)|_X
,
\] 
which holds for any $g\in C^1_X(J_j)$.
By putting here $g(\la)=\hat{f}(\la)e^{i2\pi\la x}$,
one gets
\begin{align}\label{eq:4A}
\max_{x\in[-A,A]}
\bigg|
\Exp \sum_{j=1}^m \radem_j S^{\R}_{J_j} f(x)
-
\Exp \sum_{j=1}^m \radem_j S_{I_{j,\ell}}^{(\ell)} f(x)
\bigg|_X
\le \frac K \ell \;\|f\|_{L^1_X([-A,A])} \to 0
\;\text{ as } \ell\to\infty;
\end{align}
the constant $K$ here depends on the intervals $J_j$ and
on $A$, but not on $\ell$.
By \eqref{eq:dconstants},
\[
\BNorm{\Exp\bnorm{\sum_{j=1}^m\radem_j S_{I_{j,\ell}}^{(\ell)}\,f}{X}}{L^p_X([-A,A])}
\leq C\Norm{f}{L^p_X([-A,A])}.
\]
Now \eqref{eq:3A} is obtained by
passing to the limit as $\ell$ goes to $\infty$
and taking into account \eqref{eq:4A}.

Next, one can pass to the limit
in \eqref{eq:3A} as $A\to\infty$
to get $(ii)$.
\end{proof}

Now we can give  the proof of Theorem \ref{real}. By using Theorem
\ref{toro} and Lemma \ref{paso}  we have that if a Banach space
$X$ is a UMD space with type $2$, then $X$  has $\LPReq{p}(\R)$
for all $p\in[2,\infty)$. In this case  by using  Lemma
\ref{lem:redToSmooth} and the parallel estimate to
\eqref{eq:smoothToProve} developed  in the original paper
\cite{RdF:83} we get that $X$ has $\LPReq{p}(\R,w)$ for all
$p\in[2,\infty)$ and all $w\in A_{p/2}$. For the converse implications
we just use  Lemma \ref{paso} and Theorem \ref{toro}.

\section{$L^\infty$--$\BMOtwo$ unboundedness}\label{miscelanea}

One could ask whether in the $\LPReq{\infty}(\T)$ property, the
$\exp L$ norm can be replaced by
the $\BMOtwo$ norm. Here we show that it
is not the case; moreover,
\textit{even in the one-interval inequality
\begin{equation}\label{eq:BMO} \|S_I f \|_{\BMOtwo(\T)} \le C \| f\|_{L^\infty(\T)},
\end{equation}
$C$ cannot be chosen independently on the subinterval $I\subset \Z$}.

Indeed, take the interval $I= [0, n)\subset \Z$,  it is well known that
\[
S_{[0,n)}f = \frac{i}{2}\Big( H f - e_n H ( e_{-n} f ) \Big),
\]
where $H=-i[S_{[0,\infty)}-S_{(-\infty,0)}]$ is the Hilbert transform on $\T$, and we recall our notation $e_n(x)=e^{i 2\pi nx}$.
As $H$ is bounded from $L^\infty(\T)$ into $\BMOtwo(\T)$ ,
inequality \eqref{eq:BMO} is equivalent to the existence of a
constant  $C$ such that for any $n\in \Z$, we have
\begin{equation}\label{eq:BMO3}
\| e_n H( e_{-n} f )\|_{\BMOtwo(\T)} \le C \| f\|_{L^\infty(\T)}.
\end{equation}
By choosing $f = e_n g$ with $g \in L^\infty(\T)$ we
infer that
\begin{equation}\label{eq:BMO4}
 \| e_n H(g)\|_{\BMOtwo(\T)} \le C \| g\|_{L^\infty(\T)},
\end{equation}
where $C$ should not depend on $n$.   Choose a function
$g \in L^\infty(\T)$ such that
$Hg \in \BMOtwo(\T)\setminus
L^\infty(\T)$, then next Lemma shows that a uniform estimate as in
\eqref{eq:BMO4} cannot hold.

\begin{lemma}\label{lem:BMO}
Suppose $f\in \BMOtwo(\T)$ satisfies
$\sup_{n\in\Z} \Norm{ e_n  f}{\BMOtwo(\T)}\le C<\infty $.
Then $f\in L^\infty(\T)$.
\end{lemma}

\begin{proof} We
interpret $f$ as a $1$-periodic function on $\R$.
For an interval $K$ of the real line of length less than $1$
and
a $1$-periodic function $g$ on $\R$,
put $g_K=\frac 1 {\abs{K}}\int_K g$.
One has the estimates
\[
\frac 1 {|K|}\int_K |f|=
\frac 1 {|K|}\,\int_K |fe_n|\le
\frac 1 {|K|}\,\int_K \big|fe_n-(fe_n)_K\big|\,+\,|(fe_n)_K|\le C+C=2C
\]
for a sufficiently large $|n|$, due to the Riemann-Lebesgue lemma.
It follows that $M|f|(x)\le 2C$ for any $x\in \T$, and therefore
$\Norm{f}{L^\infty(\R)}\le 2C$.
\end{proof}

The conclusion of the above  is that an $L^\infty$--$\BMOtwo$ estimate in Theorem \ref{toro} has no sense.

\end{document}